\documentclass{amsart}

\newtheorem{theorem}[equation]{Theorem}
\newtheorem{lemma}[equation]{Lemma}
\newtheorem{corollary}[equation]{Corollary}
\newtheorem{proposition}[equation]{Proposition}

\numberwithin{equation}{section}

\usepackage{amscd,amsmath,amssymb,verbatim}

\begin{document}

\title[On integrality properties of hypergeometric series]{On integrality properties of \\ hypergeometric series}
\author{Alan Adolphson}
\address{Department of Mathematics\\
Oklahoma State University\\
Stillwater, Oklahoma 74078}
\email{adolphs@math.okstate.edu}
\author{Steven Sperber}
\address{School of Mathematics\\
University of Minnesota\\
Minneapolis, Minnesota 55455}
\email{sperber@math.umn.edu}
\date{\today}
\begin{abstract}
Let $A$ be a set of $N$ vectors in ${\mathbb Z}^n$ and let $v$ be a vector in~${\mathbb C}^N$ 
that has minimal negative support for $A$.  Such a vector $v$ gives rise to a formal series 
solution of the $A$-hypergeometric system with parameter $\beta=Av$.  If $v$ lies 
in~${\mathbb Q}^n$, then this series has rational coefficients.   Let $p$ be a prime number.  
We characterize those $v$ whose coordinates are rational, $p$-integral, and lie in the closed 
interval $[-1,0]$ for which the corresponding normalized series solution has $p$-integral 
coefficients.  From this we deduce further integrality results for hypergeometric series.
\end{abstract}
\maketitle

\section{Introduction}

The $p$-integrality properties of hypergeometric series were first examined by Dwork\cite{D1,D2}, later contributions are due to Christol\cite{C} and the authors\cite{AS}.  Here we deduce further integrality properties from \cite[Theorems~1.5 and~3.5]{AS}, whose proofs are included to make this article self-contained.  Recently Delaygue\cite{D} and Delaygue-Rivoal-Roques\cite{DRR} applied certain integrality criteria for hypergeometric series in their study of mirror maps.  The integrality criterion of Theorem~5.6 generalizes those criteria.  The recent work of Franc-Gannon-Mason\cite{FGM} motivated  us to study the $p$-adic unboundedness of hypergeometric series, which is described in Section~7.

Let $A=\{ {\bf a}_1,\dots,{\bf a}_N \}\subseteq{\mathbb Z}^n$ and let $L\subseteq{\mathbb Z}^N$ 
be the lattice of relations on $A$:
\[ L = \bigg\{l=(l_1,\dots,l_N)\in{\mathbb Z}^N\;\bigg|\; \sum_{i=1}^N l_i{\bf a}_i = {\bf 0}
\bigg\}. \]
Let $\beta = (\beta_1,\dots,\beta_n)\in{\mathbb C}^n$.  The {\it $A$-hypergeometric system with 
parameter $\beta$\/} is the system of partial differential operators in $\lambda_1,\dots,
\lambda_N$ consisting of the {\it box operators\/}
\begin{equation}
\Box_l = \prod_{l_i>0} \bigg( \frac{\partial}{\partial \lambda_i}\bigg)^{l_i} - \prod_{l_i<0} \bigg( 
\frac{\partial}{\partial \lambda_i}\bigg)^{-l_i} \quad\text{for $l\in L$}
\end{equation}
and the {\it Euler\/} or {\it homogeneity\/} operators
\begin{equation}
Z_i = \sum_{j=1}^N a_{ij}\lambda_j\frac{\partial}{\partial\lambda_j} -\beta_i\quad\text{for $i=1,
\dots,n$},
\end{equation}
where ${\bf a}_j = (a_{1j},\dots,a_{nj})$.  If there is a linear form $h$ on ${\mathbb R}^n$ such 
that $h({\bf a}_i)=1$ for $i=1,\dots,N$, we call this system {\it nonconfluent}; otherwise, we 
call it {\it confluent}.

Let $v=(v_1,\dots,v_N)\in{\mathbb C}^N$.  The {\it negative support\/} of $v$ is the set
\[ {\rm nsupp}(v) = \{i\in\{1,\dots,N\} | \text{ $v_i$ is a negative integer}\}.  \]
Let
\[ L_v = \{l\in L\mid {\rm nsupp}(v+l) = {\rm nsupp}(v)\} \]
and put
\begin{equation}
\Phi_v(\lambda) = \sum_{l\in L_v} [v]_l \lambda^{v+l},
\end{equation}
where
\[ [v]_l = \prod_{i=1}^N [v_i]_{l_i} \]
and
\[ [v_i]_{l_i} = \begin{cases} 1 & \text{if $l_i=0$}, \\ \displaystyle \frac{1}{(v_i+1)(v_i+2)\cdots(v_i+l_i)} &  \text{if $l_i>0$,} \\ v_i(v_i-1)\cdots(v_i+l_i+1) & \text{if $l_i<0$.} \end{cases} \]
Note that since $l\in L_v$, we have $l_i<-v_i$ if $v_i\in{\mathbb Z}_{<0}$, so $[v_i]_{l_i}$ is always well-defined.
The vector $v$ is said to have {\it minimal negative support\/} if there is no $l\in L$ for 
which ${\rm nsupp}(v+l)$ is a proper subset of ${\rm nsupp}(v)$.  The series $\Phi_v(\lambda)$ 
is a formal solution of the system (1.1), (1.2) for $\beta=\sum_{i=1}^N v_i{\bf a}_i$ if and only 
if $v$ has minimal negative support (see Saito-Sturmfels-Takayama\cite[Proposition 3.4.13]{SST}).

Let $p$ be a prime number.  In \cite{D1,D2}, Dwork introduced the idea of normalizing 
hypergeometric series to have $p$-adic radius of convergence equal to~$1$.  This involves 
simply replacing each variable $\lambda_i$ by $\pi\lambda_i$, where $\pi$ is any uniformizer of ${\mathbb Q}_p(\zeta_p)$.  Thus $\text{ord}_p\:\pi = 1/(p-1)$.  (For further applications, Dwork chose $\pi$ to be a solution of 
$\pi^{p-1} = -p$.)  We define the $p$-adically normalized hypergeometric series to be
\begin{equation}
\Phi_{v,\pi}(\lambda) = \sum_{l\in L_v} [v]_l \pi^{\sum_{i=1}^N l_i}\lambda^{v+l} 
\quad\big(=\pi^{-v}\Phi_v(\pi\lambda)\big).
\end{equation}
Note that for nonconfluent $A$-hypergeometric systems one has $\sum_{i=1}^N l_i = 0$, so in that 
case we have $\Phi_{v,\pi}(\lambda) = \Phi_v(\lambda)$, i.e., the normalized series is just the usual 
one.  In this paper we study the $p$-integrality of the coefficients of the $\lambda^{v+l}$ in 
$\Phi_{v,\pi}(\lambda)$.

Let ${\mathbb N}$ denote the set of nonnegative integers and ${\mathbb N}_+$ the set of positive integers.  Every $t\in{\mathbb N}$ has a $p$-adic expansion
\[ t=t_0+t_1p+\cdots+t_{b-1}p^{b-1}, \quad\text{$0\leq t_j\leq p-1$ for all $j$.} \]
We define the {\it $p$-weight\/} of $t$ to be ${\rm wt}_p(t) = \sum_{j=0}^{b-1} t_j$.  
This definition is extended to vectors of nonnegative integers componentwise:
if $s=(s_1,\dots,s_N)\in{\mathbb N}^N$, define ${\rm wt}_p(s) = \sum_{i=1}^N {\rm wt}_p(s_i)$.

Let $R_p$ be the set of all $p$-integral rational vectors $(r_1,\dots,r_N)\in({\mathbb Q}\cap{
\mathbb Z}_p)^N$ satisfying $-1\leq r_i\leq 0$ for $i=1,\dots,N$.  For $r\in R_p$ choose a power 
$p^a$ such that $(1-p^a)r\in{\mathbb N}^N$ and set $s=(1-p^a)r$.  We define a {\it 
weight function\/} $w_p$ on $R_p$ by setting $w_p(r) = {\rm wt}_p(s)/a$.  The positive integer $a$ is not 
uniquely determined by~$r$ but the ratio ${\rm wt}_p(s)/a$ is independent of the choice of $a$ and 
depends only on $r$.  Note that since $0\leq (1-p^a)r_i\leq p^a-1$ for all $i$ we have $0\leq 
{\rm wt}_p(s)\leq aN(p-1)$ and $0\leq w_p(r)\leq N(p-1)$.  

We consider $A$-hypergeometric systems (1.1), (1.2) for those $\beta$ for which the set
\[ R_p(\beta) = \bigg\{r=(r_1,\dots,r_N)\in R_p\;\bigg|\; \sum_{i=1}^N r_i{\bf a}_i = \beta\bigg\} \]
is nonempty.  Define $w_p(R_p(\beta)) = \inf\{w_p(r)\mid r\in R_p(\beta)\}$.  Trivially, if $v\in R_p(\beta)$ 
then $w_p(v)\geq w_p(R_p(\beta))$.
Our first main result is the following statement.
\begin{theorem}
If $v\in R_p(\beta)$, then the series $\Phi_{v,\pi}(\lambda)$ has $p$-integral coefficients if and only if
$w_p(v)=w_p(R_p(\beta))$.  If $w_p(v)>w_p(R_p(\beta))$, then the coefficients of $\Phi_{v,\pi}(\lambda)$ are 
$p$-adically unbounded.
\end{theorem}

Note that we do not assume in Theorem 1.5 that $v$ has minimal negative support, so the series 
$\Phi_v(\lambda)$ is not necessarily a solution of the system (1.1), (1.2).

Direct application of Theorem 1.5 is limited by the fact that we do not know a general procedure for computing $w_p(R_p(\beta))$ or for determining whether there exists $v\in R_p(\beta)$ such that $w_p(v) = w_p(R_p(\beta))$.  However,  we do give a lower bound for $w_p(R_p(\beta))$ (Theorem 4.6).  In many cases of interest, one can find $v\in R_p(\beta)$ for which $w_p(v)$ equals this lower bound.  This implies by the definitions that $w_p(v)=w_p(R_p(\beta))$, hence, by Theorem~1.5, $\Phi_{v,\pi}(\lambda)$ has $p$-integral coefficients.  This leads to a useful condition for certain classical hypergeometric series to have $p$-integral coefficients (Theorem~5.6) and for the series $\Phi_v(\lambda)$ to have integral coefficients (Theorem~6.3).

In the other direction, we use Theorem 1.5 to show that if the series (1.4) fails to have $p$-integral coefficients for some prime $p$ for which $v$ is $p$-integral, then it fails to have $p$-integral coefficients for infinitely many primes.  This will imply by a generalization of a theorem of Eisenstein that if $\Phi_v(\lambda)$ is an algebraic function, then it must have $p$-integral coefficients for all primes $p$ for which $v$ is $p$-integral (see Section 7).  

\section{Proof of Theorem 1.5}

For $t\in{\mathbb N}$, set
\[ \alpha_p(t) = \frac{t-{\rm wt}_p(t)}{p-1}\in{\mathbb N}. \]
Theorem 1.5 will follow from the well-known formula
\begin{equation}
 {\rm ord}_p\: t! = \alpha_p(t). 
\end{equation}
For $k\in{\mathbb N}$, $k\leq t$, put
\[ \beta_p(t,k) = \alpha_p(t) - \alpha_p(t-k) = \frac{k-{\rm wt}_p(t) + {\rm wt}_p(t-k)}{p-1}\in{\mathbb N}. \]
Then
\begin{equation}
{\rm ord}_p\: \frac{t!}{(t-k)!} = {\rm ord}_p\: \prod_{i=0}^{k-1} (t-i) = \beta_p(t,k).
\end{equation}

We extend (2.2) to $t\in{\mathbb Z}_p$.  Write
\[ t = \sum_{i=0}^\infty t_ip^i,\quad \text{$0\leq t_i\leq p-1$ for all $i$,} \]
and set $t^{(b)} = \sum_{i=0}^{b-1} t_ip^i\in{\mathbb N}$.  Then $t\equiv t^{(b)}\pmod{p^b}$. 
\begin{lemma}
Let $k\in{\mathbb N}$ with $k\leq t^{(b)}$.  One has
\[ \prod_{i=0}^{k-1} (t-i) \equiv \prod_{i=0}^{k-1} (t^{(b)}-i)\pmod{p^{\beta_p(t^{(b)},k)+1}}. \]
\end{lemma}

\begin{proof}
Write $t=t^{(b)}+p^b\epsilon$, where $\epsilon\in{\mathbb Z}_p$.  We have
\[ \prod_{i=0}^{k-1} (t-i) =\prod_{i=0}^{k-1} (t^{(b)}-i) + \sum_{j=1}^k \sum_{0\leq i_1<\dots<i_j\leq k-1} M(i_1,\dots,i_j), \]
where
\[ M(i_1,\dots,i_j) = t^{(b)}(t^{(b)}-1)\cdots \widehat{(t^{(b)}-i_1)}\cdots \widehat{(t^{(b)}-i_j)}\cdots (t^{(b)}-k+1) (p^b\epsilon)^j. \]
We have $0<t^{(b)}-i_j<\dots<t^{(b)}-i_1\leq p^b-1$, so each of these $j$ positive integers has $p$-ordinal $<b$.  This implies that
\[ {\rm ord}_p\:M(i_1,\dots,i_j)> {\rm ord}_p\:\prod_{i=0}^{k-1} (t^{(b)}-i) = \beta_p(t^{(b)},k), \]
where the last equality follows from (2.2).  
\end{proof}

\begin{corollary}
For $t\in{\mathbb Z}_p$, $k\in{\mathbb N}$, $k\leq t^{(b)}$, one has
\[ {\rm ord}_p\: \prod_{i=0}^{k-1} (t-i) = \beta_p(t^{(b)},k). \]
\end{corollary}

We apply Corollary 2.4 to determine the $p$-divisibility of the coefficients of the series~(1.4).  For $v=(v_1,\dots,v_N)\in{\mathbb Z}_p^N$, we write the $p$-adic expansions of the $v_i$ as
\[ v_i = \sum_{j=0}^\infty v_{ij}p^j,\quad\text{$0\leq v_{ij}\leq p-1$ for all $j$,} \]
and we set
\[ v_i^{(b)} = \sum_{j=0}^{b-1} v_{ij}p^j. \]
We put $v^{(b)} = (v_1^{(b)},\dots,v_N^{(b)})\in \{0,1,\dots,p^b-1\}^N$.  
\begin{proposition}
Let $v\in{\mathbb Z}_p^N$ and let $l\in{\mathbb Z}^N$ with ${\rm nsupp}(v+l) = {\rm nsupp}(v)$.  For all sufficiently large positive integers $b$, one has 
\begin{equation}
0\leq v_i^{(b)} + l_i\leq p^b-1\quad\text{for $i=1,\dots,N$,}
\end{equation}
in which case
\begin{equation}
{\rm ord}_p\: [v]_l \pi^{\sum_{i=1}^N l_i} = \frac{{\rm wt}_p((v+l)^{(b)})-{\rm wt}_p(v^{(b)})}{p-1}.
\end{equation}
\end{proposition}

\begin{proof}
We show first that (2.7) follows from (2.6).  Fix $i$, $1\leq i\leq N$.  If $l_i<0$, then
\[ [v_i]_{l_i} = \prod_{k=0}^{-l_i-1} (v_i-k), \]
so by Corollary 2.4 
\[ {\rm ord}_p\:[v_i]_{l_i} = \beta_p(v_i^{(b)},-l_i) = \frac{-l_i -{\rm wt}_p(v_i^{(b)}) + {\rm wt}_p(v_i^{(b)}+l_i)}{p-1}. \]
If $l_i>0$, then
\[ [v_i]_{l_i} = \bigg(\prod_{k=0}^{l_i-1} (v_i+l_i-k)\bigg)^{-1}, \]
so by Corollary 2.4
\[ {\rm ord}_p\:[v_i]_{l_i} = -\beta_p((v_i+l_i)^{(b)},l_i) = \frac{-l_i + {\rm wt}_p((v_i+l_i)^{(b)}) -{\rm wt}_p(v_i^{(b)})}{p-1}. \]
Note that (2.6) implies that $(v+l)^{(b)} = v^{(b)}+l$, so
\[ {\rm ord}_p\:[v]_l = \frac{{\rm wt}_p((v+l)^{(b)}) - {\rm wt}_p(v^{(b)})-\sum_{i=1}^N l_i}{p-1}. \]
This implies (2.7).

To prove (2.6) we consider three cases.  Suppose first that $v_{ij}=0$ for all sufficiently large $j$.  Then $v_i$ is a nonnegative integer and for all large $b$ we have $v_i = v_i^{(b)}$.  If $l_i<0$, then $v_i +l_i\geq 0$ since ${\rm nsupp}(v+l) = {\rm nsupp}(v)$.  If $l_i>0$, we can guarantee that $v_i+l_i\leq p^b-1$ simply by choosing $b$ to be sufficiently large.

Suppose next that $v_{ij} = p-1$ for all sufficiently large $j$.  In this case we have for all sufficiently large $b$
\[ v_i = \sum_{j=0}^{b-1} v_{ij}p^j + p^b\sum_{j=0}^\infty (p-1)p^j = \sum_{j=0}^{b-1} v_{ij}p^j -p^b, \]
i.e., $v_i$ is a negative integer.  If $l_i<0$, we will have $v_i^{(b)}+l_i\geq 0$ for $b$ sufficiently large: since $v_{ij}=p-1$ for all sufficiently large $j$, we can make $v_i^{(b)}$ arbitrarily large by choosing $b$ large.  If $l_i>0$, then $v_i+l_i<0$ since ${\rm nsupp}(v+l) = {\rm nsupp}(v)$.  But $v_i+l_i = v_i^{(b)}-p^b+l_i$, so $v_i^{(b)}+l_i<p^b$.  

Finally, suppose that there are infinitely many $j$ for which $0<v_{ij}<p-1$.  Suppose that $l_i<0$.  Since $v_{ij}>0$ for infinitely many $j$, we can make $v_i^{(b)}$ arbitrarily large by choosing $b$ large and thus guarantee that $v_i^{(b)}+l_l\geq 0$.  Suppose that $l_i>0$.  Since $v_{ij}<p-1$ for infinitely many $j$, there are infinitely many $b$ such that
\[ v_i^{(b)}\leq \sum_{j=0}^{b-2}(p-1)p^j + (p-2)p^{b-1} = p^b-1-p^{b-1}. \]
By choosing such a $b$ large enough we can guarantee that $v_i^{(b)}+l_i\leq p^b-1$.  But when this inequality holds for one $b$, it automatically holds for all larger $b$ as well. 
\end{proof}

When $v\in R_p$, one has more control over the right-hand side of (2.7).  If $b\in{\mathbb N}$ is chosen so that $(1-p^b)v\in{\mathbb N}^N$, then $(1-p^b)v=v^{(b)}$, so
\[ {\rm wt}_p(v^{(b)}) = b w_p(v). \]
If $b$ is sufficiently large and $l\in{\mathbb Z}^N$ satisfies ${\rm nsupp}(v+l) = {\rm nsupp}(v)$, then (2.6) implies that $v+(1-p^b)^{-1}l\in R_p$ as well, so $(1-p^b)v + l = v^{(b)}+l = (v+l)^{(b)}$ and
\[ {\rm wt}_p((v+l)^{(b)}) = bw_p(v+(1-p^b)^{-1}l). \]
For $v\in R_p$, Proposition 2.5 becomes the following.
\begin{corollary}
Let $v\in R_p$ and let $l\in{\mathbb Z}^N$ with ${\rm nsupp}(v+l) = {\rm nsupp}(v)$.  Then for all sufficiently large $b\in{\mathbb N}_+$ satisfying $(1-p^b)v\in{\mathbb N}^N$ we have $v+(1-p^b)^{-1}l\in R_p$ and
\begin{equation}
{\rm ord}_p\: [v]_l\pi^{\sum_{i=1}^N l_i} = \frac{b}{p-1} \big(w_p(v+(1-p^b)^{-1}l)-w_p(v)\big).
\end{equation}
\end{corollary}

To complete the proof of Theorem 1.5 we need a lemma.
\begin{lemma}
Let $v\in R_p(\beta)$.  Then $r\in R_p(\beta)$ if and only if $r=v+(1-p^b)^{-1}l$ for some $b\in{\mathbb N}_+$ satisfying $(1-p^b)v\in{\mathbb N}^N$ and some $l\in L_v$.
\end{lemma}

\begin{proof}
If $l\in L_v$, then Corollary 2.8 implies that $v+(1-p^b)^{-1}l\in R_p(\beta)$ for sufficiently large $b$ satisfying $(1-p^b)v\in{\mathbb N}^N$.  Conversely, let $r\in R_p(\beta)$ and let $b\in{\mathbb N}_+$ be chosen so that $(1-p^b)v,(1-p^b)r\in{\mathbb N}^N$.  Put $l=(1-p^b)(r-v)$, so that $r=v+(1-p^b)^{-1}l$.  We claim that $l\in L_v$.  Since $-1\leq v_i\leq 0$ for all $i$, to show that ${\rm nsupp}(v+l) = {\rm nsupp}(v)$ we need consider only the cases $v_i=-1$ and $v_i=0$.  Suppose that $v_i=-1$ for some $i$.  Then $0\leq r_i-v_i\leq 1$ so $1-p^b\leq l_i\leq 0$ and $v_i+l_i$ is a negative integer.  If $v_i=0$ for some $i$, then $-1\leq r_i-v_i\leq 0$, so $0\leq l_i\leq p^b-1$ and $v_i+l_i\in{\mathbb N}$.
\end{proof}

\begin{proof}[Proof of Theorem $1.5$]
Fix $v\in R_p(\beta)$.  By Lemma 2.10
\begin{multline*} 
R_p(\beta) = \{v+(1-p^b)^{-1}l\mid \text{$b\in{\mathbb N}_+$, $(1-p^b)v\in{\mathbb N}^N$, $l\in L_v$,} \\ \text{and $v+(1-p^b)^{-1}l\in R_p$}\}. 
\end{multline*}
The first sentence of Theorem~1.5 then follows from Corollary~2.8.  If $w_p(v)>w_p(R_p(\beta))$, then there exists $r\in R_p(\beta)$ with $w_p(r)<w_p(v)$.  Choose $b\in{\mathbb N}_+$ such that $(1-p^b)v,(1-p^b)r\in{\mathbb N}^N$.  For each positive integer $c$, define $l^{(c)} = (1-p^{bc})(r-v)$.  The proof of Lemma~2.10 shows that $l^{(c)}\in L_v$ and Corollary~2.8 gives
\begin{align*}
{\rm ord}_p\: [v]_l\pi^{\sum_{i=1}^N l^{(c)}_i} &= \frac{bc}{p-1} \big(w_p(v+(1-p^{bc})^{-1}l^{(c)})-w_p(v)\big) \\
 &=\frac{bc}{p-1}\big(w_p(r)-w_p(v)\big).
\end{align*}
Since $w_p(r)-w_p(v)<0$ and $c$ is an arbitrary positive integer, this implies the second sentence of Theorem~1.5.
\end{proof}

\section{Elementary observations}

To avoid interrupting the discussion below, we give here some definitions and elementary results that will play a role in what follows.

Let $r$ be a rational number in the interval $[-1,0]$ and let $D$ be a positive integer such that $Dr\in{\mathbb Z}$.  Let $h$ be a positive integer with $(h,D) = 1$.  One checks that there is a unique rational number $r'\in[-1,0]$ with $Dr'\in{\mathbb Z}$ such that $r-hr'\in\{0,1,\dots,h-1\}$.  We define $\phi_{h}(r) = r'$.  We denote the $k$-fold interation of this map by $\phi_h^{(k)}$.

Similarly, let $r$ be a rational number in the interval $[0,1]$ and let $D$ be a positive integer such that $Dr\in{\mathbb Z}$.  Let $h$ be a positive integer with $(h,D) = 1$.  There is a unique rational number $r'\in [0,1]$ with $Dr'\in{\mathbb Z}$ such that $hr'-r\in \{0,1,\dots,h-1\}$.  We define $\psi_h(r) = r'$.  

These two maps are related as follows.
\begin{lemma}
Let $r$ be a rational number in the interval $[0,1]$, let $D$ be a positive integer such that $Dr\in{\mathbb Z}$, and let $h$ be a positive integer with $(h,D) = 1$.  Then 
\[ \psi_h(r) = -(\phi_h(-r))\quad\text{and}\quad \psi_h(r)-1 = \phi_h(r-1). \]
\end{lemma}

We record for later use the following proposition, whose proof is straightforward.
\begin{proposition}
Let $r\in[-1,0]$ (resp.\ $r\in[0,1]$) be a rational number, let $D$ be a positive integer for which $Dr\in{\mathbb Z}$, and let $h_1$ and $h_2$ be positive integers such that $(h_1,D) = (h_2,D) = 1$.  If $h_1\equiv h_2\pmod{D}$, then $\phi_{h_1}(r) = 
\phi_{h_2}(r)$ (resp.\ $\psi_{h_1}(r) = \psi_{h_2}(r)$).
\end{proposition}

\section{Lower bound for $w_p\big(R_p(\beta)\big)$}

Let $\Delta$ be the convex hull of the set $A\cup\{{\bf 0}\}$ and let $C(\Delta)$ be the real 
cone generated by $\Delta$.  For $\gamma\in C(\Delta)$, let $w_\Delta(\gamma)$ be the smallest 
nonnegative real number $w$ such that $w\Delta$ (the dilation of $\Delta$ by the factor $w$) 
contains $\gamma$.  It is easily seen that
\begin{equation}
w_\Delta(\gamma) = \min\bigg\{\sum_{i=1}^N t_i\:\bigg|\: (t_1,\dots,t_N)\in
({\mathbb R}_{\geq 0})^N \text{ and } \sum_{i=1}^N t_i{\bf a}_i =\gamma \bigg\}. 
\end{equation}
If $\gamma\in{\mathbb Q}^n\cap C(\Delta)$, then we may replace $({\mathbb R}_{\geq 0})^N$ by 
$({\mathbb Q}_{\geq 0})^N$ in (4.1).  For any subset $X\subseteq 
C(\Delta)$, define $w_\Delta(X) = \inf\{w(\gamma)\mid\gamma\in X\}$.

We give another formula for the weight function $w_p$ on $R_p$.  For $r$ a $p$-integral rational number in $[-1,0]$, we have defined $\phi_p(r)$ in Section~3.  This definition is extended to elements of $R_p$ componentwise:  if $r=(r_1,\dots,r_N)\in R_p$, put $\phi_p(r) = (\phi_p(r_1),\dots,\phi_p(r_N))$.  

It is clear that, as element of ${\mathbb Z}_p$, a $p$-integral rational number $r$ in $[-1,0]$ has $p$-adic expansion
\begin{equation}
r = \sum_{\mu=0}^\infty (\phi_p^{(\mu)}(r) - p\phi_p^{(\mu+1)}(r)) p^\mu.
\end{equation}
If we choose a positive integer $a$ such that $(1-p^a)r\in{\mathbb N}$, i.~e., such that $\phi_p^{(a)}(r) = r$, then
\[ (1-p^a)r = \sum_{\mu=0}^{a-1} (\phi_p^{(\mu)}(r) - p\phi_p^{(\mu+1)}(r))p^\mu. \]
Since $r=\phi_p^{(0)}(r) = \phi_p^{(a)}(r)$, this gives
\[ {\rm wt}_p((1-p^a)r) = (1-p)\sum_{\mu=0}^{a-1} \phi_p^{(\mu)}(r). \]
The definition of the weight function $w_p$ on $R_p$ then gives the desired formula: if $r=(r_1,\dots,r_N)\in R_p$ and $(1-p^a)r\in{\mathbb N}^N$, then
\begin{equation}
w_p(r) = \frac{(1-p)}{a} \sum_{\mu=0}^{a-1}\sum_{i=1}^N \phi_p^{(\mu)}(r_i).
\end{equation}

Let $\sigma_{-\beta}$ be the smallest closed face of $C(\Delta)$ that contains $-\beta$, let $T$ be the 
set of proper closed faces of $\sigma_{-\beta}$, and put
\[ \sigma^\circ_{-\beta} = \sigma_{-\beta}\setminus\bigcup_{\tau\in T} \tau. \]
By the minimality of $\sigma_{-\beta}$, $-\beta\not\in\tau$ for any $\tau\in T$, so $-\beta\in
\sigma_{-\beta}^\circ$.  
\begin{lemma}
For all $r\in R_p(\beta)$ and all $\mu$, $-\sum_{i=1}^N \phi_p^{(\mu)}(r_i){\bf a}_i\in\sigma^\circ_{-\beta}$.
\end{lemma}

\begin{proof}
Since $-\beta\in\sigma_{-\beta}$ and $-\sum_{i=1}^N r_i{\bf a}_i = -\beta$ we must have $r_i=0$ for 
all $i$ such that ${\bf a}_i\not\in \sigma_{-\beta}$.  But $r_i=0$ implies $\phi_p^{(\mu)}(r_i)=0$, so 
$-\sum_{i=1}^N \phi_p^{(\mu)}(r_i){\bf a}_i\in\sigma_{-\beta}$.  Fix $\tau\in T$.  Since $-\beta\not\in\tau$, we have $r_i\neq 0$ for some $i$ such that 
${\bf a}_i\not\in\tau$.  But $r_i\neq 0$ implies that $\phi_p^{(\mu)}(r_i)\neq 0$, so $-\sum_{i=1}^N 
\phi_p^{(\mu)}(r_i){\bf a}_i\not\in\tau$.
\end{proof}

If $r,\tilde{r}\in R_p(\beta)$, then $\sum_{i=1}^N (r_i-\tilde{r}_i){\bf a}_i = {\bf 0}$, which implies that for all~$\mu$
\[ -\sum_{i=1}^N \phi_p^{(\mu)}(r_i){\bf a}_i\equiv -\sum_{i=1}^N \phi_p^{(\mu)}(\tilde{r}_i){\bf a}_i\pmod{{\mathbb Z}A}, \] 
where ${\mathbb Z}A$ denotes the abelian group generated by $A$.  We may thus choose $\beta^{(\mu)}\in{\mathbb Q}^n$ such that 
$-\sum_{i=1}^N \phi_p^{(\mu)}(r_i){\bf a}_i\in -\beta^{(\mu)}+{\mathbb Z}A$ for all $r\in R_p(\beta)$.  
It now follows from (4.1) and Lemma 4.4 that for all $r\in R_p(\beta)$,
\[ -\sum_{i=1}^N \phi_p^{(\mu)}(r_i)\geq w_\Delta\bigg(-\sum_{i=1}^N \phi_p^{(\mu)}(r_i){\bf a}_i\bigg)\geq w_\Delta\big(\sigma_{-\beta}^\circ\cap(-\beta^{(\mu)}+{\mathbb Z}A)\big). \]
Equation (4.3) therefore gives: if $r\in R_p(\beta)$ and $(1-p^a)r\in{\mathbb N}^N$, then
\begin{equation}
w_p(r)\geq \frac{p-1}{a}\sum_{\mu=0}^{a-1} w_\Delta\big(\sigma_{-\beta}^\circ\cap(-\beta^{(\mu)}+{\mathbb Z}A)\big).
\end{equation}

Note that the right-hand side of (4.5) is independent of the choice of $a$: if $e$ is the smallest positive 
integer such that $-\beta^{(e)}+{\mathbb Z}A = -\beta+{\mathbb Z}A$, then $e$ divides $a$ and 
one has
\[ \sum_{\mu=0}^{a-1} w_\Delta\big(\sigma_{-\beta}^\circ\cap(-\beta^{(\mu)}+{\mathbb Z}A)\big) = \frac{a}{e} 
\sum_{\mu=0}^{e-1} w_\Delta\big(\sigma_{-\beta}^\circ\cap(-\beta^{(\mu)}+{\mathbb Z}A)\big). \]
Equation (4.5) now shows that
\[ w_p(r)\geq \frac{p-1}{e}\sum_{\mu=0}^{e-1} w_\Delta\big(\sigma_{-\beta}^\circ\cap(-\beta^{(\mu)}+
{\mathbb Z}A)\big) \]
for all $r\in R_p(\beta)$, which establishes the following result.
\begin{theorem}
With the above notation, we have
\[ w_p(R_p(\beta))\geq \frac{p-1}{e}\sum_{\mu=0}^{e-1} w_\Delta\big(\sigma_{-\beta}^\circ\cap(-\beta^{(\mu)}+
{\mathbb Z}A)\big). \]
\end{theorem}

The following corollary is now an immediate consequence of Theorem~1.5.
\begin{corollary}
If $v\in R_p(\beta)$ satisfies 
\begin{equation}
w_p(v) = \frac{p-1}{e}\sum_{\mu=0}^{e-1} w_\Delta\big(\sigma_{-\beta}^\circ\cap
(-\beta^{(\mu)}+{\mathbb Z}A)\big), 
\end{equation}
then the series $\Phi_{v,\pi}(\lambda)$ has $p$-integral coefficients.
\end{corollary}

We make Corollary 4.7 more precise.  Let $v\in R_p(\beta)$ with $(1-p^a)v\in{\mathbb N}^N$.  Set
\[ \beta_\mu = \sum_{i=1}^N \phi_p^{(\mu)}(v_i){\bf a}_i. \]
Then $-\beta_\mu\in -\beta^{(\mu)}+{\mathbb Z}A$, so
\[ -\sum_{i=1}^N \phi_p^{(\mu)}(v_i)\geq w_\Delta(-\beta_\mu)\geq w_\Delta\big(\sigma_{-\beta}^\circ\cap
(-\beta^{(\mu)}+{\mathbb Z}A)\big), \]
where the first inequality follows from (4.1).  It follows from (4.3) that (4.8) can hold only when
\[ -\sum_{i=1}^N \phi_p^{(\mu)}(v_i) = w_\Delta\big(\sigma_{-\beta}^\circ\cap
(-\beta^{(\mu)}+{\mathbb Z}A)\big)\quad\text{for $\mu=0,1,\dots,a-1$.} \]
\begin{corollary}
Let $v\in R_p(\beta)$ satisfy $(1-p^a)v\in{\mathbb N}^N$.  If
\begin{equation}
 -\sum_{i=1}^N \phi^{(\mu)}_p(v_i) = w_\Delta\big(\sigma_{-\beta}^\circ\cap(-\beta^{(\mu)}+{\mathbb Z}A)\big) 
\end{equation}
for $\mu=0,1,\dots,a-1$, then the series $\Phi_{v,\pi}(\lambda)$ has $p$-integral coefficients.
\end{corollary}

Let $D\in{\mathbb N}$ satisfy $Dv\in{\mathbb Z}^N$.  Note that, except for the factor $(p-1)$, the right-hand side of (4.8) is independent of $p$.  And by Proposition~3.2 the right-hand side of (4.3), except for the factor $(1-p)$, depends only on $p\pmod{D}$.  We thus have the following result.
\begin{proposition}
When (4.8) holds for one prime $p$, it also holds for all primes congruent to $p$ modulo~$D$.  In this case, the series $\Phi_{v,\pi'}(\lambda)$ has $p'$-integral coefficients for all primes $p'$ congruent to $p$ modulo $D$ (where $\text{ord}_{p'}\:\pi' = 1/(p'-1)$).
\end{proposition}

\section{Classical hypergeometric series}

To illustrate the information contained in Corollary 4.9, we derive an integrality condition for certain classical hypergeometric series.  For $z\in{\mathbb C}$, $k\in{\mathbb N}$, we use the Pochhammer notation $(z)_k$ for the increasing factorial $(z)_k = z(z+1)\cdots (z+k-1)$.

Let $c_{js}, d_{ks}\in{\mathbb N}$, $1\leq j\leq J$, $1\leq k\leq K$, $1\leq s\leq r$, and let
\begin{align}
C_j(x_1,\dots,x_r) &= \sum_{s=1}^r c_{js}x_s, \\
D_k(x_1,\dots,x_r) &= \sum_{s=1}^r d_{ks}x_s. 
\end{align}
To avoid trivial cases, we assume that no $C_j$ or $D_k$ is identically zero.  We also assume that for each $s$, some $c_{js}\neq 0$ or some $d_{ks}\neq 0$, i.~e., each variable $x_s$ appears in some $C_j$ or $D_k$ with nonzero coefficient.  We make the hypothesis that
\begin{equation}
\sum_{j=1}^J C_j(x_1,\dots,x_r) = \sum_{k=1}^K D_k(x_1,\dots,x_r),
\end{equation}
i.~e.,
\begin{equation}
\sum_{j=1}^J c_{js} = \sum_{k=1}^K d_{ks} \quad\text{for $s=1,\dots,r$.}
\end{equation}

Let $\Theta = (\theta_1,\dots,\theta_J)$ and $\Sigma = (\sigma_1,\dots,\sigma_K)$ be sequences of $p$-integral rational numbers in the interval $(0,1]$.  Consider the series
\begin{equation}
F(t_1,\dots,t_r) = \sum_{m_1,\dots,m_r = 0}^\infty \frac{(\theta_1)_{C_1(m)}\cdots(\theta_J)_{C_J(m)}}{(\sigma_1)_{D_1(m)}\cdots(\sigma_K)_{D_K(m)}} t_1^{m_1}\cdots t_r^{m_r}.
\end{equation}
These series are all $A$-hypergeometric (see \cite{AS2}) and include, for example, the four $r$-variable Lauricella series $F_A,F_B,F_C,F_D$.
For a nonnegative integer $\mu$ and a positive integer $h$ that does not divide the denominator of any $\theta_j$ or $\sigma_k$ we put
\[ \psi_h^{(\mu)}(\Theta) = (\psi_h^{(\mu)}(\theta_1),\dots,\psi_h^{(\mu)}(\theta_J)) \text{ and }
\psi_h^{(\mu)}(\Sigma) = (\psi_h^{(\mu)}(\sigma_1),\dots,\psi_h^{(\mu)}(\sigma_K)). \]
We will apply Corollary 4.9 to obtain a condition for $F(t)$ to have $p$-integral coefficients.  Define a step function on ${\mathbb R}^r$ (analogous to that of Landau\cite{L})
\begin{multline*} 
\xi(\Theta,\Sigma;x_1,\dots,x_r) = \\ 
\sum_{j=1}^J \lfloor 1-\theta_j + C_j(x_1,\dots,x_r)\rfloor - \sum_{k=1}^K \lfloor 1-\sigma_k + D_k(x_1,\dots,x_r)\rfloor. 
\end{multline*}

\begin{theorem}
Let $D$ be a positive integer such that $D\theta_j,D\sigma_k\in{\mathbb N}$ for all $j,k$.  Let $h$ be a positive integer prime to $D$ and choose $a$ such that $h^a\equiv 1\pmod{D}$.  
The series (5.5) has $p$-integral coefficients for all primes $p\equiv h\pmod{D}$ if 
\[ \xi\big(\psi_h^{(\mu)}(\Theta),\psi_h^{(\mu)}(\Sigma);x_1,\dots,x_r\big) \geq 0 \]
for $\mu=0,1,\dots,a-1$ and all $x_1,\dots,x_r\in [0,1)$.
\end{theorem}

\begin{corollary}
If the hypothesis of Theorem 5.6 is satisfied for positive integers $h_1,\dots,h_{\varphi(D)}$ representing all the residue classes modulo $D$ that are prime to $D$, then $F(t)$ has $p$-integral coefficients for all primes $p$ not dividing $D$.
\end{corollary}

{\bf Remark.}  It follows from the proof of Delaygue-Rivoal-Roques\cite[Proposition 1]{DRR} that when the hypothesis of Corollary 5.7 is satisfied there exist positive integers $b_1,\dots,b_m$ such that $F(b_1t_1,\dots,b_mt_m)$ has integral coefficients.  

We do not know whether the converse of Corollary 5.7 is true, but it does hold in some special cases.  When $\theta_j=1$ for all $j$ and $\sigma_k=1$ for all $k$, the coefficients of $F(t)$ are ratios of products of factorials and one may take $a=1$.  In this case Theorem 5.6 reduces to one direction of a result of Landau\cite{L}, who also proved the converse.  The converse of Corollary 5.7 holds as well in the one-variable case when $J=K$ and
\begin{equation}
F(t) = \sum_{m=0}^\infty \frac{(\theta_1)_m\cdots(\theta_J)_m}{(\sigma_1)_m\cdots(\sigma_J)_m}t^m 
\end{equation}
by a theorem of Christol\cite{C} (see also the discussion in \cite{DRR}).  

Beukers and Heckman\cite[Theorem 4.8]{BH} have characterized those series (5.8) that are algebraic functions.  It follows from their characterization that when (5.8) is an algebraic function the hypothesis of Corollary 5.7 is satisfied, so the coefficients of (5.8) are $p$-integral for all primes $p$ not dividing $D$.  Using Theorem 1.5, we show later (Proposition~7.5) that the same conclusion holds more generally for the series (1.3) when it is an algebraic function.

The $A$-hypergeometric system associated to the series (5.5) is the one introduced in \cite{AS2}.  Put $n= r+J+K$.  Let ${\bf a}_1,\dots,{\bf a}_n$ be the standard unit basis vectors in ${\mathbb R}^n$ and for $s=1,\dots,r$ let
\[ {\bf a}_{n+s} = (0,\dots,0,1,0,\dots,0,c_{1s},\dots,c_{Js},-d_{1s},\dots,-d_{Ks}), \]
where the first $r$ coordinates have a $1$ in the $s$-th position and zeros elsewhere.  Our hypothesis that some $c_{js}$ or some $d_{ks}$ is nonzero implies that ${\bf a}_1,\dots,{\bf a}_{n+r}$ are all distinct.  Put $N=n+r$ and let $A = \{{\bf a}_i\}_{i=1}^{N}\subseteq {\mathbb Z}^n$.   Let $\Delta$ be the convex hull of $A\cup\{{\bf 0}\}$.  
Under (5.4) the elements of the set $A$ all lie on the hyperplane $\sum_{i=1}^n u_i =1$ in ${\mathbb R}^n$, so the weight function $w_\Delta$ of (4.1) is given by
\begin{equation}
w_\Delta(\gamma_1,\dots,\gamma_n) = \sum_{i=1}^n \gamma_i.
\end{equation}
We choose 
\[ v=(-1,\dots,-1,-\theta_1,\dots,-\theta_J,\sigma_1-1,\dots,\sigma_K-1,0,\dots,0)\in {\mathbb Q}^N, \]
where $-1$ and $0$ are repeated $r$ times, so that
\[ \beta = \sum_{i=1}^N v_i{\bf a}_i = (-1,\dots,-1,-\theta_1,\dots,-\theta_J,\sigma_1-1,\dots,\sigma_K-1)\in{\mathbb R}^n. \]

As noted in \cite{AS2} one has
\begin{multline}
L = \{ l=(-m_1,\dots,-m_r,-C_1(m),\dots,-C_J(m), \\
D_1(m),\dots,D_K(m),m_1,\dots,m_r)\mid m=(m_1,\dots,m_r)\in{\mathbb Z}^r\}
\end{multline}
and
\begin{multline}
L_v = \{ l=(-m_1,\dots,-m_r,-C_1(m),\dots,-C_J(m), \\
D_1(m),\dots,D_K(m),m_1,\dots,m_r)\mid m=(m_1,\dots,m_r)\in{\mathbb N}^r\}.
\end{multline}
After a calculation one sees that the series (1.3) is
\begin{multline}
\Phi_v(\lambda) = \\
 (\lambda_1\cdots\lambda_{r+J})^{-1}
\sum_{m_1,\dots,m_r=0}^\infty  \frac{\displaystyle\prod_{j=1}^J (\theta_j)_{C_j(m)}}{\displaystyle\prod_{k=1}^K (\sigma_k)_{D_k(m)}} \frac{\displaystyle\prod_{k=1}^K \lambda_{r+J+k}^{D_k(m)}\prod_{s=1}^r \lambda_{n+s}^{m_s}}{\displaystyle\prod_{s=1}^r (-\lambda_s)^{m_s}\prod_{j=1}^J (-\lambda_{r+j})^{C_j(m)}}.
\end{multline}
Furthermore, the hypothesis (5.4) implies that we are in the nonconfluent case, so there is no normalizing factor of $\pi$ to include.  We can thus apply Corollary~4.9 to get an integrality condition for the series (5.12) and hence an integrality condition for the series (5.5).  

Using Lemma 3.1 and the definition of $v$, we see that the left-hand side of (4.10) is given by
\begin{equation}
r+\sum_{j=1}^J \psi^{(\mu)}_p(\theta_j) + \sum_{k=1}^K \psi^{(\mu)}_p(1-\sigma_k).
\end{equation}
Theorem 5.6 then follows from Corollary 4.9, Proposition 3.2, and the following proposition.  Note that ${\mathbb Z}A = {\mathbb Z}^n$ in this case.
\begin{proposition}
For each $\mu$ one has
\begin{equation}
w_\Delta\big(\sigma^\circ_{-\beta}\cap(-\beta^{(\mu)}+{\mathbb Z}^n)\big)= r+\sum_{j=1}^J \psi^{(\mu)}_p(\theta_j) + \sum_{k=1}^K \psi^{(\mu)}_p(1-\sigma_k)
\end{equation}
if and only if
\begin{equation}
\xi\big(\psi_p^{(\mu)}(\Theta),\psi_p^{(\mu)}(\Sigma);x_1,\dots,x_r\big) \geq 0\quad
\text{for all $x_1,\dots,x_r\in[0,1)$.}
\end{equation}
\end{proposition}

\begin{proof}
It suffices to prove this when $\mu=0$, the other cases being analogous in view of Lemma~3.1.  By (5.9) $w_\Delta(-\beta)$ equals the right-hand side of (5.15) when $\mu=0$, so we are reduced to proving that
\begin{equation}
w_\Delta\big(\sigma^\circ_{-\beta}\cap(-\beta+{\mathbb Z}^n)\big) = w_\Delta(-\beta)
\end{equation}
if and only if
\begin{equation}
\sum_{j=1}^J \lfloor 1-\theta_j + C_j(x_1,\dots,x_r)\rfloor - \sum_{k=1}^K \lfloor 1-\sigma_k + D_k(x_1,\dots,x_r)\rfloor\geq 0 
\end{equation}
for all $x_1,\dots,x_r\in[0,1)$.

The set $A$ and the associated polytope $\Delta$ and cone $C(\Delta)$ were discussed in~\cite{AS2}.  Since $\theta_j>0$ for all $j$, it follows from \cite[Lemma 2.5]{AS2} that $-\beta$ is an interior point of~$C(\Delta)$, hence $\sigma^\circ_{-\beta} = C(\Delta)^\circ$,  the interior of~$C(\Delta)$.  Equation (5.17) thus becomes
\[ w_\Delta\big(C(\Delta)^\circ\cap(-\beta+{\mathbb Z}^n)\big) = w_\Delta(-\beta). \]
Fix $u\in{\mathbb Z}^n$ such that $-\beta+u$ is an interior point of $C(\Delta)$ with
\begin{equation}
w_\Delta(-\beta+u) = w_\Delta\big(C(\Delta)^\circ\cap(-\beta+{\mathbb Z}^n)\big).
\end{equation}
Then (5.17) is equivalent to
\[ w_\Delta(-\beta+u) = w_\Delta(-\beta). \]
The inequality $w_\Delta(-\beta + u)\leq w_\Delta(-\beta)$ is trivial from (5.19), so we are reduced to showing that the inequality
\[ w_\Delta(-\beta+u)\geq w_\Delta(-\beta) \]
holds if and only if (5.18) holds for all $x_1,\dots,x_r\in[0,1)$.  By (5.9) we have 
\[ w_\Delta(-\beta+u) = w_\Delta(-\beta) + \sum_{i=1}^n u_i, \]
so this is equivalent to showing that, for $u$ satisfying (5.19),
\begin{equation}
\sum_{i=1}^n u_i\geq 0
\end{equation}
if and only if (5.18) holds for all $x_1,\dots,x_r\in[0,1)$.

By \cite[Lemma 2.4]{AS2} we may write
\begin{equation} 
-\beta+u = \sum_{i=1}^N z_i{\bf a}_i
\end{equation}
with $z_i\geq 0$ for all $i$ and $z_i>0$ for $i=1,\dots,r+J$.  Note that since the coordinates of each ${\bf a}_i$ sum to 1, Equation (5.9) implies that 
\begin{equation}
w_\Delta(-\beta+u) = \sum_{i=1}^N z_i.
\end{equation}

We must have $z_i\leq 1$ for all $i$.  For if some $z_{i_0}>1$, then
\begin{equation}
 -\beta + u-{\bf a}_{i_0} = (z_{i_0}-1){\bf a}_ {i_0} + \sum_{\substack{i=1\\ i\neq i_0}}^N z_i{\bf a}_i 
\end{equation}
is an element of $-\beta + {\mathbb Z}^n$ interior to  $C(\Delta)$ since every ${\bf a}_i$ with $>0$ coefficient in (5.21) occurs with $>0$ coefficient in (5.23).  But $w_\Delta(-\beta+u-{\bf a}_{i_0})<w_\Delta(-\beta+u)$ by (5.22), contradicting (5.19).

We claim that $z_i<1$ for $i=r+J+l$, $l=1,\dots,N$.  If $z_{i_0}=1$ for some $i_0\in\{r+J+1,\dots,N\}$, then (5.23) becomes
\[ -\beta + u-{\bf a}_{i_0} = \sum_{\substack{i=1\\ i\neq i_0}}^N z_i{\bf a}_i. \]
But since $z_i>0$ for $i=1,\dots,r+J$, the point $-\beta + u-{\bf a}_{i_0}$ is an element of $-\beta + {\mathbb Z}^n$ interior to $C(\Delta)$ by \cite[Lemma 2.5]{AS2}, and again $w_\Delta(-\beta+u-{\bf a}_{i_0})<w_\Delta(-\beta+u)$, contradicting (5.19).  

We have proved that in the representation (5.21) one has 
\begin{equation}
z_i\in(0,1]\quad\text{for $i=1,\dots,r+J$} 
\end{equation}
and
\begin{equation}
z_i\in [0,1)\quad\text{for $i=r+J+1,\dots,N$.} 
\end{equation}

We now examine (5.21) coordinatewise.  For $s=1,\dots,r$ we have
\begin{equation}
 1+u_s = z_s + z_{n+s}. 
\end{equation}
By (5.24) and (5.25) we have $z_s\in(0,1]$ and $z_{n+s}\in[0,1)$.  Since $u_s\in{\mathbb Z}$, Equation~(5.26) implies
\begin{equation}
u_s=0 \quad\text{for $s=1,\dots,r$}
\end{equation}
and
\begin{equation}
z_s=1-z_{n+s} \quad\text{for $s=1,\dots,r$.}
\end{equation}

For $j=1,\dots,J$ we have
\begin{equation}
 \theta_j+u_{r+j} = z_{r+j}+C_j(z_{n+1},\dots,z_{n+r}). 
\end{equation}
Since $u_{r+j}\in{\mathbb Z}$ and $z_{r+j}\in(0,1]$ we have
\begin{equation}
z_{r+j} = 
1 + \big\lfloor -\theta_j+ C_j(z_{n+1},\dots,z_{n+r})\big\rfloor - \big( -\theta_j + C_j(z_{n+1},\dots,z_{n+r})\big) 
\end{equation}
for $j=1,\dots,J$, which implies by (5.29)
\begin{equation}
u_{r+j}= 1 + \big\lfloor -\theta_j + C_j(z_{n+1},\dots,z_{n+r})\big\rfloor\quad\text{for $j=1,\dots, J$.}
\end{equation}

For $k=1,\dots,K$ we have
\begin{equation}
 1-\sigma_k + u_{r+J+k} = z_{r+J+k} -D_k(z_{n+1},\dots,z_{n+r}). 
\end{equation}
Since $u_{r+J+k}\in{\mathbb Z}$ and $z_{r+J+k}\in[0,1)$ we have
\begin{equation}
z_{r+J+k} = 
1-\sigma_k + D_k(z_{n+1},\dots,z_{n+r}) - \big\lfloor 1-\sigma_k +D_k(z_{n+1},\dots,z_{n+r})\big\rfloor
\end{equation}
for $k=1,\dots,K$, which implies by (5.32)
\begin{equation}
u_{r+J+k} = - \big\lfloor 1-\sigma_k+D_k(z_{n+1},\dots,z_{n+r})\big\rfloor\quad\text{for $k=1,\dots,K$.}
\end{equation}

Adding (5.27), (5.31), and (5.34) gives
\begin{equation}
\sum_{i=1}^n u_i = 
\sum_{j=1}^J \big\lfloor 1-\theta_j + C_j(z_{n+1},\dots,z_{n+r})\big\rfloor - \sum_{k=1}^K
\big\lfloor 1-\sigma_k+D_k(z_{n+1},\dots,z_{n+r})\big\rfloor.
\end{equation}
This shows that if (5.18) holds for all $x_1,\dots,x_r\in[0,1)$, then (5.20) holds.  

The argument can be reversed to show that (5.20) implies (5.18).  If there exist $x_1,\dots,x_r\in[0,1)$ for which (5.18) fails, define $z_{n+s} = x_s$ for $s=1,\dots,r$, define $z_s$ for $s=1,\dots,r$ by (5.28), define $z_{r+j}$ for $j=1,\dots,J$ by (5.30), and define $z_{r+J+k}$ for $k=1,\dots,K$ by (5.33).  The right-hand side of Equation (5.21) then defines an element of $C(\Delta)$.  It equals the left-hand side of (5.21) with $u$ given by Equations (5.27), (5.31), and (5.34), so it is also an element of $-\beta + {\mathbb Z}^n$.  Furthermore, (5.35) holds, so if (5.18) fails, then (5.20) fails also.
\end{proof}

The argument above closely follows the proof of \cite[Theorem 2.1(a)]{AS2}.  One can also derive an analogue of \cite[Theorem 2.1(b)]{AS2} by following that proof.
Define ${\mathcal D}(\Theta,\Sigma)\subseteq [0,1)^r$ to be the subset where
\[ 1-\theta_j + C_j(x_1,\dots,x_r)\geq 1\quad\text{for some $j$} \]
or
\[ 1-\sigma_k + D_k(x_1,\dots,x_r)\geq 1\quad\text{for some $k$.} \]
Clearly $\xi(\Theta,\Sigma;x_1,\dots,x_r)=0$ on the complement of ${\mathcal D}(\Theta,\Sigma)$.  
\begin{proposition}
The point $-\beta$ is the unique interior point of $C(\Delta)$ satisfying
\[ w_\Delta(-\beta) = w_\Delta\big(C(\Delta)^\circ\cap(-\beta + {\mathbb Z}^n)\big)\] 
if and only if $\xi(\Theta,\Sigma;x_1,\dots,x_r)\geq 1$ for all $x\in {\mathcal D}(\Theta,\Sigma)$.
\end{proposition}

\section{Integral coefficients}

Corollary 4.9 leads to a simple condition for the series $\Phi_v(\lambda)$ to have integral coefficients.  We suppose that we are in the nonconfluent case, so the set $A$ lies on a hyperplane $h(u) = 1$ in ${\mathbb R}^n$ and there is no normalizing factor of $\pi$.  We also suppose that the coordinates of $v$ equal either 0 or $-1$, say,
\[ v=(-1,\dots,-1,0,\dots,0) \]
where $-1$ is repeated $M$ times and 0 is repeated $N-M$ times.  We then have
\begin{multline*}
 L_v = \{l=(l_1,\dots,l_N)\in L\mid\text{$l_i\leq 0$ for $i=1,\dots,M$} \\
\text{and $l_j\geq 0$ for $j=M+1,\dots,N$}\}.
\end{multline*}
The series (1.3) becomes
\begin{equation}
\Phi_v(\lambda) = (\lambda_1\cdots \lambda_M)^{-1}\sum_{l\in L_v} (-1)^{\sum_{i=1}^M l_i}\frac{\prod_{i=1}^M (-l_i)!}{\prod_{j=M+1}^N l_j!}\lambda^l.
\end{equation}
In this case we have $\beta=-\sum_{i=1}^M {\bf a}_i \in{\mathbb Z}A$ and Equation (4.10) becomes 
\begin{equation}
M=w_\Delta\big(\sigma^\circ_{-\beta}\cap{\mathbb Z}A\big).
\end{equation}
Since $h({\bf a}_i)=1$ for $i=1,\dots,N$, it follows from (4.1) that $w_\Delta(\gamma) = h(\gamma)$ for $\gamma\in C(\Delta)$.  In particular, $w_\Delta(-\beta) = M$.  Equation (6.2) thus asserts that on the set $\sigma^\circ_{-\beta}\cap{\mathbb Z}A$, the function $w_\Delta (=h)$ assumes its minimum value at the point $-\beta$.  Since $v$ lies in $R_p(\beta)$ for all primes $p$, we can apply Corollary 4.9 to obtain the following conclusion.
\begin{theorem}
With the above notation, if $w_\Delta\big(\sigma^\circ_{-\beta}\cap{\mathbb Z}A\big)=M$, then the series (6.1) has integral coefficients.
\end{theorem}

{\bf Example.}  Let
\begin{equation}
f_\lambda=\sum_{i=1}^N \lambda_jx^{{\bf b}_j} \in K[\lambda_1,\dots,\lambda_N][x_0,\dots,x_n] 
\end{equation}
be a homogeneous polynomial of degree $d$, where $\lambda_1,\dots,\lambda_N$ are indeterminates and $K$ is a field.  Put ${\bf a}_j = ({\bf b}_j,1)$ and let $A=\{{\bf a}_i\}_{i=1}^N\subseteq{\mathbb Z}^{n+2}$.  We make the hypothesis that $d$ divides $n+1$, say, $n+1=de$.  Note that this generalizes the case of Calabi-Yau hypersurfaces in ${\mathbb P}^n$ where $d=n+1$.  The points of $A$ all lie on the hyperplane $\sum_{i=0}^n u_i = du_{n+1}$ in ${\mathbb R}^{n+2}$.  We suppose also that there exist monomials $x^{{\bf b}_1},\dots,x^{{\bf b}_e}$ such that
\[ x^{{\bf b}_1}\cdots x^{{\bf b}_e} = x_0\cdots x_n. \]
Take $\beta = (-1,\dots,-1,-e)$.  Then $-\beta$ is the unique point of $\sigma_{-\beta}^\circ\cap{\mathbb Z}A$ where $w_\Delta$ assumes its minimum value, namely, $w_\Delta(-\beta) = e$.  Set
\[ v=(-1,\dots,-1,0,\dots,0), \]
where $-1$ is repeated $e$ times and $0$ is repeated $N-e$ times.  Then
\[ \sum_{i=1}^N v_i{\bf a}_i = -{\bf a}_1-\cdots-{\bf a}_e = (-1,\dots,-1,-e) = \beta. \]
Since (6.2) holds in this case (with $M=e$), Theorem 6.3 implies that the resulting series (6.1) has integral coefficients.
In \cite{AS1}, we used this fact to give a $p$-adic analytic formula for the (unique) reciprocal root of the zeta function of $f_\lambda = 0$ of minimal $p$-divisibility when $K$ is a finite field. 

\section{Unboundedness of hypergeometric series}

Let $\Phi_{v,\pi}(\lambda)$ be as in (1.4).  If this series does not have $p$-integral coefficients, then by Theorem 1.5 there exists $v'\in R_p(\beta)$ for which $w_p(v')<w_p(v)$.  Equation~(4.3) then implies that if $(1-p^a)v,(1-p^a)v'\in {\mathbb N}^N$, then
\begin{equation}
\sum_{\mu=0}^{a-1} \sum_{i=1}^N \phi_p^{(\mu)}(v'_i)< \sum_{\mu=0}^{a-1} \sum_{i=1}^N \phi_p^{(\mu)}(v_i).
\end{equation}
\begin{proposition}
Let $v,v'\in R_p(\beta)$ with $w_p(v')<w_p(v)$ and let $D$ be a positive integer such that $Dv,Dv'\in{\mathbb N}^N$.  Let $p'$ be any prime such that $p'\equiv p\pmod{D}$ and let $\pi'$ satisfy $\text{ord}_{p'}\:\pi' = 1/(p'-1)$.   Then the series $\Phi_{v,\pi'}(\lambda)$ does not have $p'$-integral coefficients and is $p'$-adically unbounded.
\end{proposition}

\begin{proof}
Since $p\equiv p'\pmod{D}$ we have $(1-(p')^a)v,(1-(p')^a)v'\in {\mathbb N}^N$.  By (7.1) and Proposition 3.2
\[ \sum_{\mu=0}^{a-1} \sum_{i=1}^N \phi_{p'}^{(\mu)}(v'_i)< \sum_{\mu=0}^{a-1} \sum_{i=1}^N \phi_{p'}^{(\mu)}(v_i). \]
Equation (4.3) then implies that $w_{p'}(v')<w_{p'}(v)$, so the assertion of the proposition follows from Theorem~1.5.
\end{proof}

Since there are infinitely many primes $p'$ such that $p'\equiv p\pmod{D}$, we get the following corollary.
\begin{corollary}
If the series $\Phi_{v,\pi}(\lambda)$ does not have $p$-integral coefficients, then there are infinitely many primes $p'$ for which the series $\Phi_{v,\pi'}(\lambda)$ does not have $p'$-integral coefficients and is $p'$-adically unbounded.
\end{corollary}

In the nonconfluent case, this corollary gives the following result.
\begin{corollary}
If the series $\Phi_v(\lambda)$ of (1.3) has $p$-integral coefficients for all but finitely many primes $p$, then it has $p$-integral coefficients for all primes $p$ for which $v$ is $p$-integral.
\end{corollary}

We apply this corollary to the case of hypergeometric series which are also algebraic functions.
\begin{proposition}
Suppose that $\Phi_v(\lambda)$ is algebraic over ${\mathbb Q}(\lambda)$ and the set $L_v$ lies in a cone in ${\mathbb R}^N$ with vertex at the origin.  Then $\Phi_v(\lambda)$ has $p$-integral coefficients for all primes $p$ for which $v$ is $p$-integral.
\end{proposition}

\begin{proof}
By the multivariable version of Eisenstein's theorem on algebraic functions (see Sibuya-Sperber\cite[Theorem 37.2]{SS} and the remark below), there exist positive integers $b_0,b_1,\dots,b_N$ such that the series $b_0\Phi_v(b_1\lambda_1,\dots,b_N\lambda_N)$ has integral coefficients.  This implies that $\Phi_v(\lambda)$ has $p$-integral coefficients for all but finitely many primes.  The assertion of the proposition then follows from Corollary~7.4.
\end{proof}

{\bf Remark.}  Since the reference \cite{SS} seems not to be readily available, we include a proof of the multivariable version of Eisenstein's Theorem in the next section.

{\bf Example.}  Proposition 7.5 implies that if the series $F(t_1,\dots,t_r)$ defined in (5.5) is an algebraic function, then its coefficients are $p$-integral for all primes $p$ for which all $\theta_j$ and $\sigma_k$ are $p$-integral.  In particular, if $\theta_j=1$ and $\sigma_k=1$ for all $j$ and $k$ (so that the coefficients of $F(t_1,\dots,t_r)$ are ratios of products of factorials) then these coefficients are integers.  When $r=1$, this observation is due to Rodriguez-Villegas\cite{RV}.

\section{Multivariable Eisenstein's Theorem}

The notation of this section is independent of the rest of the paper.  We leave it to the reader to prove (for example, by induction on $n$) that if $C$ is a cone in ${\mathbb R}^n$ generated by vectors in ${\mathbb Z}^n$ and with a vertex at the origin, then there exists a basis $E$ for ${\mathbb Z}^n$ such that every element of $C$ is a linear combination of elements of $E$ with nonnegative coefficients.  
This implies that the ${\mathbb Q}$-algebra of formal power series of the form 
\[ \sum_{u\in{\mathbb Z}^n\cap C} c_uX^u,\quad c_u\in{\mathbb Q}, \]
is a ${\mathbb Q}$-subalgebra of a power series ring isomorphic to ${\mathbb Q}[[X_1,\dots,X_n]]$.  Thus for the version of Eisenstein's theorem used in the proof of Proposition~7.5, it suffices to prove it for ${\mathbb Q}[[X_1,\dots,X_n]]$.

We recapitulate an old result of Sibuya-Sperber\cite{SS} which gives an analogue in the setting of several variables of Eisenstein's classical one-variable result giving necessary conditions on the coefficients of a power series $g(X) \in \mathbb{Q}[[X]]$ in order for it to be an algebraic function over $\mathbb{Q}(X)$. This argument is taken entirely from the account in \cite{SS}. We repeat it here for convenience since the original publication may not be easily accessible.  

We will proceed more generally and prove the following multivariable analogue of Eisenstein as a corollary.
\begin{theorem}
Let $g(X_1,\dots,X_n) = \sum_{|\alpha| \geq 0} C(\alpha)X^{\alpha} \in \mathbb{Q}[[X_1, \dots,X_n]]$ be algebraic over $\mathbb{Q}(X_1,\dots,X_n)$. Then there exists a positive integer $N$ such that
\begin{equation}
N^{|\alpha|}C(\alpha) \in {\mathbb Z}                           
\end{equation}
for all $\alpha$ with $|\alpha| >0$.
\end{theorem}

This result is an easy consequence of the following somewhat stronger result.
\begin{theorem}
Let $A$ be a commutative integral domain with identity of characteristic zero and let $K$ be its field of fractions. Let $f(X) = \sum_{m=0}^{\infty} c_m X^m \in K[[X]]$ be algebraic over $K(X)$. Then there exist positive integers $M$ and $M'$ with $M\geq M'$  such that the power series
\[ \tilde{f}(X) = \sum_{m=M+1}^{\infty} c_m X^{m-M'} \]
satisfies an equation of the form
\[ \rho \tilde{f}(X) = X F_0(X, \tilde{f}), \]
where $\rho \in A$ and  $F_0 \in A[X,Z]$.
\end{theorem}

Before turning to the proof of Theorem 8.3 we show how Theorem 8.1 follows from Theorem 8.3. 
\begin{proof}[Proof of Theorem 8.1]
Let $g(X_1,\dots,X_n)$ be as in the statement of Theorem 8.1. Let $\{ t, Y_1,\dots,Y_n \}$ be a collection of $n+1$ variables and set
\[ f(t,Y_1,\dots,Y_n) = g(tY_1,\dots,tY_n) = \sum_{m=0}^{\infty} \tilde{c}_m t^m, \]
where $\tilde{c}_m = \sum_{|\alpha| = m} C(\alpha) Y^{\alpha}$.  
Since we assume $g$ to be algebraic over the field $\mathbb{Q}(X_1,\dots,X_n)$ it follows that $f$ is algebraic over $\mathbb{Q}(t, Y_1,\dots,Y_n)$.  To apply Theorem~8.3, we take $A= \mathbb{Z}[Y_1,\dots,Y_n]$ with field of fractions $K = \mathbb{Q}(Y_1,\dots,Y_n)$. Then according to Theorem 8.3 there are positive integers $M$ and $M'$ with $M \geq M'$ such that $\tilde{f}(t) = \sum_{m=M+1}^{\infty} \tilde{c}_m t^{m-M'}$ satisfies 
\begin{equation}
\rho \tilde{f}(t) = tF_0(t, \tilde{f}),                     
\end{equation}
where $\rho \in \mathbb{Z}[Y_1,\dots,Y_n]$ and $F_0 \in \mathbb{Z}[Y_1,\dots,Y_n][t,Z]$.  We rewrite $\tilde{f} = \sum_{m=0}^{\infty} \gamma_m t^m$, (so $\gamma_{m} = \tilde{c}_{m+M'}$ for $m\geq M+1-M'$ and $\gamma_m = 0$ for $m=0,\dots,M-M'$) and $F_0(t, \tilde{f}) = \sum_{ (j,h) \in J} \mu_{j,h}t^j \tilde{f}^h$ with $\gamma_m \in \mathbb{Q}[Y_1,\dots,Y_n]$, $J$ a finite subset of ${\mathbb N}^2$, and $\mu_{j ,h} \in \mathbb{Z}[Y_1,\dots,Y_n]$. 

The first possibly nonzero coefficient of $\tilde{f}$ is $\gamma_{M+1-M'}$, and from (8.4) we compute
\[ \rho\gamma_{M+1-M'} = \mu_{M-M',0}\in{\mathbb Z}[Y_1,\dots,Y_n]. \]
For general $m$, identifying coefficients of $t^m$ in Equation~(8.4) and multiplying by $\rho^{m-1}$ gives the recursions
\[ \rho^m \gamma_m = \sum_{\substack{(j,h)\in J\\ \sigma_1 + \cdots +\sigma_h = m-1-j}} \mu_{j,h} \rho^j (\rho^{\sigma_1} \gamma_{\sigma_1})(\rho^{\sigma_2} \gamma_{\sigma_2})\cdots(\rho^{\sigma_h} \gamma_{\sigma_h}). \]
This implies by induction on $m$ that $\rho^m \gamma_m \in \mathbb{Z}[Y_1,\dots,Y_n]$ for all $m>0$. 

We write $\rho = \tau\hat{\rho}$, where $\hat{\rho} \in \mathbb{Z}[Y_1,\dots,Y_n]$ has Gauss content equal to 1 and $\tau \in \mathbb{Z}$.  Also for each $m$ we write $\gamma_m = \nu_m\hat{\gamma}_m$ with $\hat{\gamma}_m \in \mathbb{Z}[Y_1,\dots,Y_n]$ having Gauss content equal to 1 and $\nu_m\in \mathbb{Q}$.  Clearly then by unique factorization in $\mathbb{Z}[Y_1,\dots,Y_n]$ we have $\tau^m \nu_m \in \mathbb{Z}$.  Thus $\tau^m$ clears denominators for $\gamma_m = \tilde{c}_{m+M'}$ for $m\geq M+1-M'$, so taking provisionally $N=\tau$ in (8.2) above we see that (8.2) holds for all $\alpha$ with $|\alpha| \geq M+1$. Multiplying $N$ by a suitable factor to clear denominators in the coefficients of the lower degree terms completes the proof.
\end{proof}

\begin{proof}[Proof of Theorem 8.3]
We are given that $Z =f$ satisfies a non-trivial polynomial equation $F(X, Z) = 0$ where $F$ is a polynomial in $X$ and $Z$ with coefficients in $K$.  We fix a choice of $F$ of minimal degree in $Z$.  If we put $F_Z = {\partial F}/{\partial Z}$, then clearly $F_Z(X, f) \neq 0$.  We may write 
\[ F_Z(X, f(X))= X^{\mu} \phi(X), \]
where $\mu \in{\mathbb N}$ and $\phi(0) \neq 0$.  Set $M= 2\mu +1$ and $M' = \mu + 1$.  Let $f_M(X) = \sum_{m=0}^M c_mX^m$.  Since $M > \mu$, we may write 
\[ F_Z(X, f_M(X)) = X^{\mu} \phi_M(X), \]
where $\phi_M(X) \in K[X]$ and $\phi_M(0) \neq 0$.  Also, since $F(X, f(X)) = 0$ and $f \equiv f_M \pmod{X^{M+1}}$, we have $F(X, f_M(X)) = X^{M+1} \psi_M(X)$, where $\psi_M(X) \in K[X]$.  From the Taylor series we get
\[ F(X,f_M(X)+W) = F(X, f_M(X)) + F_Z(X, f_M(X))W + W^2G_M(X,W) \]
with $G_M \in K[X, W]$.  But then for $\tilde{f}$ defined as in Theorem~8.3,
\begin{multline*}
F(X, f_M(X) + X^{M'} \tilde{f}) = \\ 
X^{M+1} \psi_M(X) + X^{\mu + M'}\phi_M(X) \tilde{f} + X^{2M'}\tilde{f}^2G_M(X, X^{M'}\tilde{f}). 
\end{multline*}
Since $f = f_M(X) + X^{M'}\tilde{f},$ multiplying throughout by $X^{-(\mu +M')}$ gives
\begin{align*}
0 & = X^{-(\mu + M')} F(X, f_M + X^{M'}\tilde{f}) \\
& = \phi_M(X)\tilde{f} + X^{M+1-\mu - M'} \psi_M(X)  + X^{M' - \mu}\tilde{f}^2G_M(X,X^{M'}\tilde{f})\\
& =\phi_M(0)\tilde{f} + X\bigg( X^{-1}(\phi_M(X) - \phi_M(0))\tilde{f} + X^{M - \mu - M'}\psi_M(X) \\
 & \hspace*{.25in}+ X^{M' - \mu - 1}\tilde{f}^2G_M(X, X^{M'}\tilde{f}) \bigg). 
\end{align*}
It follows then that
\begin{equation}
\phi_M(0)\tilde{f} = XF_1(X, \tilde{f})                  
\end{equation}
with (using the definitions of $M$ and $M'$ above)
\[ F_1(X, Z) = - \bigg( X^{-1}(\phi_M(X) - \phi_M(0))Z +  \psi_M(X) + Z^2G_M(X, X^{\mu+1}Z) \bigg). \]
This shows that $F_1 \in K[X,Z]$.  Multiplying by a suitable element of $A$ to clear fractions on both sides of (8.5) gives Theorem~8.3.
\end{proof}

\end{document}